\documentclass[12pt]{amsart}     
\usepackage{amsmath,amssymb,amsthm}     
\usepackage{color,colordvi,graphicx}   
\usepackage{latexsym}            
\usepackage{hyperref}     
\usepackage{multirow}     
\newtheorem{theorem}{Theorem}

\newtheorem{proposition}[theorem]{Proposition}     

\theoremstyle{definition}  
\newtheorem{definition}[theorem]{Definition}     
\newtheorem{example}[theorem]{Example}     
\newtheorem{remark}[theorem]{Remark}     

\newcounter{FNC}[page]
\def\fauxfootnote#1{{\addtocounter{FNC}{2}$\Magenta{^\fnsymbol{FNC}}$%
     \let\thefootnote\relax\footnotetext{\Magenta{$^\fnsymbol{FNC}$#1}}}}
\newcommand{\CC}{{\mathbb C}}
\newcommand{\GG}{{\mathbb G}}

\newcommand{\PP}{{\mathbb P}}
\newcommand{\QQ}{{\mathbb Q}}
\newcommand{\ZZ}{{\mathbb Z}}

\newcommand{\calE}{{\mathcal E}}
\newcommand{\calO}{{\mathcal O}}

\newcommand{\alg}{{\rm alg}}
\newcommand{\id}{{\it id}}

\newcommand{\defcolor}[1]{\RoyalBlue{#1}}
\newcommand{\demph}[1]{\defcolor{{\sl #1}}}

\DeclareMathOperator{\cl}{{\rm cl}}
\DeclareMathOperator{\rk}{{\rm rk}}

\DeclareMathOperator{\Hom}{{\rm Hom}}
\DeclareMathOperator{\Rat}{{\rm Rat}}
\DeclareMathOperator{\Alg}{{\rm Alg}}
\DeclareMathOperator{\Num}{{\rm Num}}

\title[General witness sets for numerical algebraic geometry]{General witness sets for numerical algebraic geometry}
\author[F.~Sottile]{Frank Sottile}     
\address{Frank Sottile\\     
         Department of Mathematics\\     
         Texas A\&M University\\     
         College Station\\     
         Texas \ 77843\\     
         USA}     
\email{sottile@math.tamu.edu}     
\urladdr{\url{http://www.math.tamu.edu/~sottile}}   
\thanks{Research of Sottile supported in part by NSF grant DMS-1501370 and
  Simons Foundation Collaboration Grant for Mathematics Number 636314.}
\subjclass{65H10, 14C17, 14M15}
\keywords{numerical algebraic geometry, intersection theory, witness set, Schubert variety}      
\begin{document}     
     
\begin{abstract}
  Numerical algebraic geometry has a close relationship to intersection theory from algebraic geometry.
  We deepen this relationship, explaining how rational or algebraic equivalence gives a homotopy.
  We present a general notion of witness set for
  subvarieties of a smooth complete complex algebraic variety using ideas from intersection theory.
  Under appropriate assumptions, general witness sets enable numerical algorithms such as sampling and membership.
  These assumptions hold for products of flag manifolds.
  We introduce Schubert witness sets, which provide general witness sets for Grassmannians and flag manifolds.
\end{abstract}       

\maketitle     
%
\section*{Introduction}\label{Sec:intro}     
    
Numerical algebraic geometry uses numerical analysis to study algebraic varieties.
Its foundations rest on numerical homotopy continuation, which enables
the numerical computation of solutions to systems of polynomial equations~\cite{SW05}.
It relies on the fundamental concept of a 
witness set~\cite{INAG,NAG}, which is a data structure for representing a
subvariety of affine or projective space on a computer.
Witness sets also appear in symbolic computation under the term lifting fiber~\cite{GH}.

A witness set for an irreducible variety $V$ of dimension $k$ is a
triple, $(F,\Lambda,W)$, where $F$ is a system of polynomial equations whose zero set contains
$V$ as a component and $\Lambda$ is a general linear space of codimension $k$ (represented by $k$ general
linear polynomials) which meets $V$ transversally in the finite set $W$ of points.
Numerical continuation of the points $W$ when $\Lambda$ is moved allows one to, for example,
sample points from $V$.
Consequently, $W$ may be considered to be a generic point of $V$ in the sense of Weil~\cite{Weil}.

A witness set for a subvariety also represents its fundamental cycle in homology.
The homology of projective space has a basis given by classes 
$[L]$ of linear spaces.
Since linear spaces satisfy duality---$L\cap\Lambda$ is a point when $L$
and $\Lambda$ are general linear spaces of complementary dimension---the homology class $[V]$ of
a subvariety $V$ of dimension $k$ is determined by its degree, which is the number of points in its intersection
with a general linear space $\Lambda$ of codimension $k$.
That is, if $L$ is a linear space of dimension $k$, then
\[
   [V]\ =\  \deg(V\cap\Lambda)\cdot [L]\,.
\]
In a witness set, we replace the number $\deg(V\cap\Lambda)$ by the set $W:=V\cap\Lambda$ and
require that the intersection be transverse, which we may, by Bertini's Theorem.

The concept of witness sets and their manipulation is linked to
ideas from intersection theory~\cite{Fu84a,Fu84b}.
A witness set $W$ is a concrete representation of the localized intersection product
$[V]\bullet[\Lambda]\in H_0 (V\cap\Lambda)$~\cite[Ch.~8]{Fu84a}.
As $W$ is a set of $\deg(V)$ points of $V$, we are implicitly working in the group of cycles modulo numerical
equivalence.
As a homotopy is a family of varieties (or points) over $\CC$, homotopies are connected to the notion
of rational equivalence.

We propose a notion of witness set for subvarieties of a smooth algebraic variety $X$,
based on ideas from intersection theory.
This requires an equivalence relation, such as numerical equivalence, on algebraic cycles such that the resulting group of
cycles on $X$  is a finitely generated free abelian group on which the intersection pairing is nondegenerate. 
Choosing an additive basis of cycles gives general witness sets for subvarieties of $X$.
With additional assumptions (see \S~\ref{S:Gwitness}) this notion is refined, and there are algorithms using general
witness sets such as changing a witness set, sampling, and membership testing.

Products of projective spaces satisfy these additional assumptions, and 
these ideas for such products were proposed in~\cite{HR15}.
These assumptions hold for flag manifolds, where the natural
general witness sets are Schubert witness sets. 
We explain how Schubert witness sets enable numerical continuation
algorithms for sampling and membership.

Numerical algebraic geometry operates on the geometric side of algebraic geometry, with
algorithms based on geometric constructions, such as fiber products~\cite{SW08}, images of
maps~\cite{HS10}, and monodromy~\cite[\S15.4]{SW05}.
It is also suited for intersection theory, using excess intersection formulas to compute Chern
numbers~\cite{DEPS}. 
Understanding witness sets in terms of intersection theory is a natural continuation.

This paper is organized as follows.
Section~\ref{S:NAG} gives background from numerical algebraic geometry, including numerical continuation,
witness sets, and some fundamental algorithms.
Section~\ref{S:Intersection} gives background from intersection theory
and explains the connection of rational equivalence to numerical homotopy continuation. 
We present general witness sets in Section~\ref{S:Gwitness}, and explain
how additional hypotheses enable algorithms for sampling and membership.
In Section~\ref{S:Schubert} we introduce Schubert witness sets,
which are the natural general witness sets for flag manifolds and explain 
the fundamental algorithms for Schubert witness sets.

%

\section{Classical witness sets}\label{S:NAG}

We review aspects of numerical algebraic geometry as may be found in~\cite{BertiniBook,SW05}.

\subsection{Homotopy continuation}\label{SS:homotopy}

A \demph{homotopy} is a polynomial map
 \begin{equation}\label{Eq:homotopy}
   H\ =\ H(x;t)\ \colon\ \CC^n\times\CC\ \longrightarrow\ \CC^N\,,
 \end{equation}
where $H^{-1}(0)\subset\CC^n\times\CC$ defines an algebraic curve \defcolor{$C$} with a dominant projection to
the distinguished ($t$) coordinate, $\CC$.
We suppose that $1$ is a regular value of the projection to $\CC$ and we know the points of the fiber, and
we use them to obtain the points of the fiber over $0$.

For example, suppose that $F=(f_1,\dotsc,f_n)$ with $f_i$ a polynomial of degree $d_i$.
Then the \demph{B\'ezout homotopy} 
\[
   H(x;t)\ :=\ (1-t)F\ +\ t( x_i^{d_i}{-}1\mid i=1,\dotsc,n)
\]
connects the points over $t=1$, $(x_1,\dotsc,x_n)$
where $x_i$ is a $d_i$-th root of unity, to the
unknown solutions to $F=0$.

Restricting $t$ to the interval $[0,1]\subset\CC$ (more generally to a path in $\CC$ connecting 1 to
0~\cite[\S~2.1]{BertiniBook}), the algebraic curve $C$ becomes a collection of real 
paths in $\CC^n\times[0,1]$ connecting points in the fiber at $t=1$ to those at $t=0$.
A point $(x,1)$ in the fiber of $C$ at $t=1$ lies on a unique path, and standard predictor-corrector
methods construct a sequence $(x_1,t_1),\dotsc,(x_s,t_s)$ of points along that path with $1>t_1>\dotsb>t_s=0$ so
that $x_s$ is a solution to $H(x,0)=0$.
This is illustrated in Figure~\ref{F:one}.
%
%
\begin{figure}[htb]

 \begin{picture}(420,125)
    \put(0,10){\includegraphics{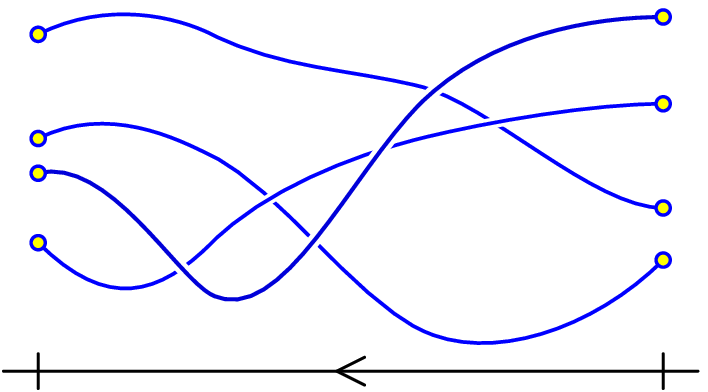}}
      \put(7,0){$0$} \put(189,0){$1$} \put(100,0){$t$}
       \put(90,115){$C|_{[0,1]}$}
    \put(220,10){\includegraphics{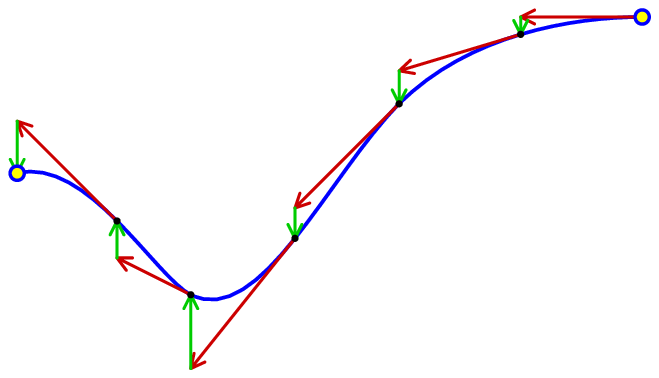}}
 \end{picture}

 \caption{Paths over $[0,1]$ and predictor-corrector steps.}\label{F:one}
\end{figure}

These numerical algorithms do not compute points on paths or on varieties, but rather refinable approximations to
such points.
This uses Newton's method which replaces a point $x\in\CC^n$ by the result of
a Newton step
\[
   \defcolor{N_F(x)}\ :=\ x\ -\ (DF(x))^{-1}(x)\,,
\]
where $F\colon \CC^n\to\CC^n$ is a polynomial map and $DF(x)$ is its Jacobian derivative.
When $x$ is sufficiently close to a solution $x^*$ of $F$, the sequence of iterations defined by $x_0:=x$
and $x_i=N_F(x_{i-1})$ for $i>0$ satisfies
\[
   \|x_i-x^*\|\ <\ \left(\frac{1}{2}\right)^{2^i-1}\|x-x^*\|\,.
\]
When this occurs, we say that $x$ \demph{converges quadratically} to $x^*$.
Smale's $\alpha$-theory~\cite{Smale86} involves a computable~\cite{alphaC} constant $\alpha(F,x)$ 
such that if $\alpha(F,x)<(13-3\sqrt{17})/4$, then $x$ converges quadratically to a solution.
Other approaches to certification (e.g.\ Krawcyzk's Method~\cite{Kraw,MJ77}) use interval arithmetic~\cite{MKRC09}.

We ignore the question of whether our approximations lie in the basin of quadratic convergence under
Newton iterations and simply refer to them as solutions, state that they lie on paths or on varieties, {\it et cetera}.

\subsection{Witness sets and algorithms}\label{SS:witness}
Algorithms based on numerical homotopy continuation can compute the isolated solutions to a system of
polynomial equations $F(x)=0$ and follow solutions along homotopy paths.
Sommese, Verschelde, and Wampler~\cite{INAG,NAG} introduced the notion of a witness set, which enables the representation
and manipulation of algebraic subvarieties of $\CC^n$ using numerical homotopy continuation.

Let $F\colon\CC^n\to\CC^N$ be a polynomial map and $V\subset\CC^n$ a union of components of $F^{-1}(0)$ of 
dimension $k$.
A \demph{witness set} for $V$ is a triple $(F,\Lambda,W)$ where $\Lambda\colon\CC^n\to\CC^k$ is $k$
general affine forms and $W=V\cap\Lambda^{-1}(0)$.
As $\Lambda$ is general, $W$ consists of $\deg(V)$ points and the intersection is transverse.

We may use a witness set $W$ to compute other witness sets.
If $\Lambda'$ is another set of $k$ independent affine forms, the convex combination
$\Lambda(t):= (1-t)\Lambda'+t\Lambda$ may be used with $F$ to define a homotopy that connects the points
$W$ at $t=1$ to points $W':=V\cap (\Lambda')^{-1}(0)$ at $t=0$.
Numerical continuation along this homotopy computes the points $W'$ (when finite) from the
points $W$.
When $\Lambda'$ is general, we obtain another witness set $(F,\Lambda',W')$ for $V$.

As every point of $V$ lies on some affine subspace of codimension $k$ which meets $V$ properly, continuation of a
witness set along such homotopies samples points of $V$.
Moreover, if $p\in\CC^n$ and we choose $\Lambda'$ such that $\Lambda'(p)=0$, but $\Lambda'$ is otherwise general, then
$p\in V$ if and only if $p\in W'$.
These three algorithms, moving a witness set, sampling points of a variety, and the membership test, are 
fundamental methods to study a variety $V$ given a witness set, and form the basis for more sophisticated
algorithms. 

\section{Intersection theory}\label{S:Intersection}

We recall aspects of algebraic cycles and intersection theory, and then discuss how rational equivalence leads to
homotopies as in \S\S\ref{SS:homotopy}. 
This material, with proofs is found in Chapters 1 and 19  of~\cite{Fu84a}.
Other sources include~\cite{EH_3264,Fu84b}.

Let $X$ be a smooth algebraic variety of dimension $n$.
If $V,\Lambda\subset X$ are subvarieties of dimensions $k$ and $l$ with $k{+}l\geq n$, then either $V\cap\Lambda$ is empty
or it has dimension at least $k{+}l{-}n$.
It is  \demph{proper} if it has this expected dimension.
The intersection  $V\cap\Lambda$ is \demph{transverse} at a point $p\in V\cap\Lambda$ when both $V$ and $\Lambda$ are
smooth at $p$ and their tangent spaces at $p$ span the tangent space of $X$ at $p$,
$T_pV+T_p\Lambda=T_pX$.
The intersection $Y\cap Z$ is \demph{generically transverse} if it is transverse at a dense set of points
$p\in V\cap\Lambda$. 
Generically transverse is necessary as any of $V$, $\Lambda$, or $V\cap\Lambda$ may have singular points.
Generically transverse intersections are proper.

\subsection{Intersection theories}\label{SS:intersection}

Let $X$ be a connected, complete, smooth, irreducible complex algebraic variety of dimension $n$.
For each $0\leq k\leq n$, the group \defcolor{$Z_kX$} of $k$-cycles on $X$ is the free
abelian group generated by the $k$-dimensional irreducible subvarieties of $X$. 
The  \demph{fundamental cycle $[V]$} of an irreducible subvariety $V$ of $X$ is the corresponding generator of 
$Z_k X$.
A subscheme $V\subset X$ of dimension $k$ also has a fundamental cycle.
For each irreducible component $\Lambda$ of $V$ of dimension $k$, let \defcolor{$m_{\Lambda,V}$} be its multiplicity in
$V$, which is the generic multiplicity of $V$ along $\Lambda$. 
The fundamental cycle of $V$ is 
\[
  \defcolor{[V]}\ :=\ \sum_\Lambda m_{\Lambda,V} [\Lambda]\,.
\]
A cycle $\sum \alpha_V[V]$ with $\alpha_V\geq 0$ is \demph{effective}.
If $\alpha_V\in\{0,1\}$, it is \demph{multiplicity-free}.
The fundamental cycle of a generically transverse intersection is multiplicity-free.
A map $\iota\colon Y\to X$ of varieties induces a map $\iota_*\colon Z_k Y\to Z_k X$.
When $\iota$ is an inclusion and $V\subset Y$ is a subscheme, $\iota_*[V]=[\iota(V)]$.

Sending a subvariety $V$ to its fundamental cycle in homology induces the \demph{cycle class map} 
$\defcolor{\cl}\colon Z_kX\to H_{2k}(X,\QQ)$.
Its kernel is the group \defcolor{$\Hom_k X$} of $k$-cycles with an integer multiple homologically equivalent to
zero and its image is the $k$-th \demph{algebraic homology $H_{k}^{\alg}X$} of $X$.
(The shift in homological degree from $2k$ to $k$ is for notational consistency.)
The group $H_{k}^{\alg}X$ is a finitely generated free abelian group.
As $X$ is smooth, homology has an intersection product which induces a bilinear map 
$H_{k}^{\alg}X\times H_{l}^{\alg}X\to H_{k+l-n}^{\alg}X$, where $(\alpha,\beta)\mapsto \alpha\cdot\beta$. 
When $k+l=n$, this  gives the \demph{intersection pairing} $H_{n-k}^{\alg}X\times H_{k}^{\alg}X\to H_{0}^{\alg}X=\ZZ$.

Suppose that $Y\subset X\times\PP^1$ is an irreducible subvariety of dimension $k{+}1$ with projections
$\iota$ to $X$ and $f$ to $\PP^1$,
 \begin{equation}\label{Eq:TwoProjections}
   \raisebox{-23pt}{\begin{picture}(84,46)(0,0)
     \put(19,39){$Y\subset X\times\PP^1$}
     \put(20,36){\vector(-1,-2){12.5}}  \put(24,36){\vector(1,-2){12.5}}
     \put( 7,22){$\iota$}     \put(33,22){$f$}
     \put( 0,0){$X$}  \put(33,0){$\PP^1$}
   \end{picture}}
 \end{equation}
%
%
%
where $f$ is surjective.
The fibers of $f$ are naturally subschemes of $X$ of dimension $k$.
Call the cycle
 \begin{equation}\label{Eq:ERE}
    [\iota(f^{-1}(0))]\ -\  [\iota(f^{-1}(1))]\ \in\ Z_k X
 \end{equation}
an \demph{elementary rational equivalence}.
Elementary rational equivalences generate the subgroup $\defcolor{\Rat_kX}\subset Z_kX$ of $k$-cycles
rationally equivalent to zero.
The quotient $\defcolor{A_kX}:=Z_kX/\Rat_kX$ is the $k$-th \demph{Chow group}  of $X$.
As $X$ is smooth, there is an intersection product as with homology.

Let $V$ and $\Lambda$ be subvarieties of $X$ of dimension $k$ and $l$.
The localized intersection product~\cite[Ch.~8]{Fu84a} of their fundamental cycles
is a cycle class
 \begin{equation}\label{Eq:LIP}
   \defcolor{[V]\bullet[\Lambda]}\ \in\ A_{k+l-n} (V\cap\Lambda)\,.
 \end{equation}
Its image in $A_{k+l-n} X$ under the map induced by the inclusion 
$V\cap\Lambda\hookrightarrow X$ is the intersection product $[V]\cdot[\Lambda]$.
When the intersection is proper, the localized intersection product is the
fundamental cycle of the scheme-theoretic intersection, $[V\cap\Lambda]$.

Let $\defcolor{\deg}\colon A_0X\to \ZZ$ be the degree map on 0-cycles 
\[
    \deg\ \colon\ \sum m_p [p]\ \longmapsto\ \sum m_p\,,
\]
the sum over $p\in X$.
Note that only finitely many coefficients $m_p$ are non-zero.
Composing with the product gives an intersection pairing $A_{n-k}X\times A_kX\to A_0X \to \ZZ$ as before.

If we replace $\PP^1$ by an irreducible curve $T$ and $0,1\in\PP^1$ by two smooth points of $T$ in the
definition of rational equivalence, we obtain \demph{algebraic equivalence}.
Let $\defcolor{\Alg_kX}\subset Z_k X$ be the group generated by differences of algebraically equivalent $k$-cycles,
the group of cycles algebraically equivalent to zero.
Let $\defcolor{B_kX}:=Z_kX/\Alg_kX$ be the group of cycles modulo algebraic equivalence.
This has an intersection product and pairing $B_{n-k}X\times B_{k}X \to B_0X=\ZZ$ as before.

A cycle $\beta\in Z_kX$ is \demph{numerically equivalent to zero} if, for every
$\alpha\in A_{n-k}X$, we have $\deg(\alpha\cdot \overline{\beta})=0$, where $\overline{\beta}$ is the image of $\beta$ in
$A_kX$. 
Let $\defcolor{\Num_kX}\subset Z_kX$ be the subgroup of $k$-cycles numerically equivalent to zero.
Set $\defcolor{N_kX}:=Z_kX/\Num_kX$, which is a finitely generated free abelian group.
The intersection pairing  is nondegenerate by the definition of numerical equivalence.

\begin{proposition}\label{P:intersectionTheory}
 For every $0\leq k\leq n$ we have 
 \begin{equation}\label{Eq:Cycle_relations}
   \Rat_k X\ \subset\ \Alg_kX\ \subset\ \Hom_kX\ \subset\ \Num_k X\,,
 \end{equation}
 as subgroups of $Z_kX$.
 The maps $A_kX\to B_kX\to H^\alg_kX\to N_kX$ are compatible with the intersection product.
 The groups $H^\alg_kX$ and $N_kX$ are finitely generated free abelian groups and the intersection pairing
 $N_{n-k} X \times N_{k}X\to\ZZ$ is nondegenerate.
\end{proposition}

Define \defcolor{$A_*X$} to be the direct sum of the $A_kX$ and the same for \defcolor{$B_*X$},
\defcolor{$H^\alg_*$}, and \defcolor{$N_*X$}.

\begin{remark}
 The first two inclusions in~\eqref{Eq:Cycle_relations} are strict in general.
 A conjectured equality of $\Hom_kX$ and $\Num_k X$ is a consequence of Grothendieck's `standard
 conjectures'~\cite[\S~5]{Kl94}. 
 The question of when the intersection pairing on $N_*X$ is perfect is related to the representability of integral
 homology classes by algebraic cycles.
 \hfill$\diamond$
\end{remark}

\subsection{Intersection theory  and homotopy continuation}

Elementary rational and algebraic equivalences give
homotopies in the sense of \S\S\ref{SS:homotopy}. 

Let $Y\subset X\times\PP^1$ be an irreducible subvariety of dimension $n{-}k{+}1$ having projections
$\iota$ to $X$ and $f$ to $\PP^1$ with $f$ surjective as in~\eqref{Eq:TwoProjections}.
This gives an elementary rational equivalence~\eqref{Eq:ERE} in $\Rat_{n-k}X$.
Suppose that $V\subset X$ has dimension $k$ and meets $\iota f^{-1}(1)$ transversally.
Then $(V\times\PP^1)\cap Y$ contains a curve $C$ passing through $V\cap \iota f^{-1}(1)$.
Writing $g$ for the restriction of $f$ to $C$ gives a surjective map $g\colon C\to\PP^1$.
Then $g^{-1}[0,1]$ gives arcs in $C$ connecting the points of $g^{-1}(1)$ to $g^{-1}(0)$ as in Figure~\ref{F:one}.
%
%
\begin{figure}[htb]
 \begin{picture}(390,115)
   \put(5,75){$X$}  \put(0,45){$f$}   \put(10,68){\vector(0,-1){40}}
   \put(5,11){$\PP^1$}
   \put(20,0){\begin{picture}(170,110)
     \put(0,10){\includegraphics{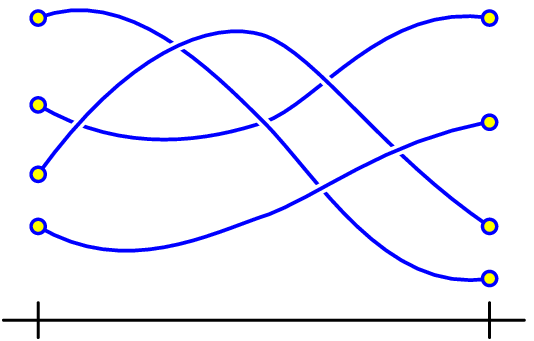}}
     \put(7,0){$0$} \put(139,0){$1$} 
     \put(79,101){$C|_{[0,1]}$}
   \end{picture}}
   \put(175,60){\vector(1,0){40}}
   \put(191,65){$\varphi$}
   \put(220,0){\begin{picture}(170,110)
     \put(0,10){\includegraphics{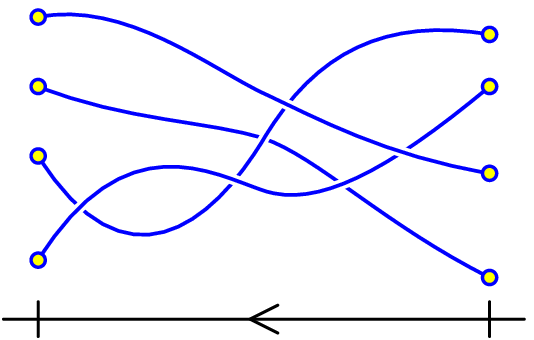}}
     \put(7,0){$0$} \put(139,0){$1$} \put(74,0){$t$}
     \put(64,99){$D|_{[0,1]}$}
   \end{picture}}
   \put(376,75){$\CC^d$}
   \put(381,68){\vector(0,-1){40}} \put(383,45){$h$}
   \put(376,11){$\CC$}
 \end{picture}
 \caption{An elementary rational equivalence defines a homotopy.}
 \label{F:ERE}
\end{figure}
Choosing  coordinates and equations for the varieties, we obtain a homotopy as in \S\S\ref{SS:homotopy}.

\begin{theorem}\label{Th:RatEquivHomotopy}
 Let $Y\subset X\times\PP^1$ give an elementary rational equivalence in $\Rat_{n-k}X$~\eqref{Eq:ERE} and
 $V\subset X$ be a subvariety of dimension $k$ meeting $\iota f^{-1}(1)$ transversally with $g$ the restriction of
 $f$ to the curve $C=(V\times\PP^1)\cap Y$. 
 Let $U\subset X$ be an affine open set containing $\iota g^{-1}[0,1]$.
 For any embedding $\varphi\colon U\to\CC^d$ there is a homotopy $H(x;t)$ defining a
 curve $D\subset\CC^d\times\CC$ with  $\varphi^{-1}(D)=C\cap(U\times\PP^1)$ and 
 $\varphi^{-1}(D|_{[0,1]})=g^{-1}[0,1]$.
\end{theorem}

\begin{proof}
 Let $\CC\subset\PP^1$ be an affine line containing $[0,1]$.
 Then the arcs $g^{-1}[0,1]$ lie in the curve $C^\circ:=C\cap(U\times\CC)$.
 Let $\varphi\colon U\to\CC^d$ be a map realizing $U$ as a subvariety of $\CC^d$.
 Then $\varphi\times\id_{\CC}$ realizes $C^\circ$ as a subvariety of $\CC^d\times\CC$.
 Let $D$ be its closure and $h\colon D\to\CC$ the projection map.
 Choosing any system of equations $H\colon\CC^d\times\CC\to\CC^N$ with $D=H^{-1}(0)$ gives a homotopy.
 See Figure~\ref{F:ERE}.
\end{proof}

\begin{remark}
 This leads to a numerical homotopy algorithm to find the points of $\iota g^{-1}(0)$, given those of 
 $\iota g^{-1}(1)$.
 Write $h\colon D\to\CC$ for the projection.
 As $\varphi \iota g^{-1}(1) = h^{-1}(1)$, we may use the homotopy to trace these points along the arcs of
 $h^{-1}[0,1]$ to obtain the points of $h^{-1}(0)$.
 Since $h^{-1}(0)=\varphi \iota g^{-1}(0)$, applying $\varphi^{-1}$ to  $h^{-1}(0)$ gives $\iota g^{-1}(0)$.
 \hfill$\diamond$
\end{remark}

\begin{remark}\label{R:AlgEquiv}
 Theorem~\ref{Th:RatEquivHomotopy} used rational equivalence as homotopy continuation assumes that 
 $t$ is rational ($t\in\CC$).
 Given an elementary algebraic equivalence, replace $\CC$ by a smooth affine curve
 $T$, the points $0$ and $1$ by points $p,q\in T$, and the interval $[0,1]$ by an arc $\gamma$ on $T$
 connecting $p$ to $q$.
 This gives arcs connecting points of $(V\times\PP^1)\cap Y$ above $p$ to points above $q$.
 Choosing coordinates $(\varphi)$ gives a homotopy $H(x;t)$, but the parameter $t$ is not
 rational, as it takes values in $\gamma\subset T$.
 This becomes a traditional homotopy by choosing a map $\psi\colon T\to\PP^1$
 with $\psi(p)=1$ and $\psi(q)=0$, and then the path $\gamma$ from $p$ to $q$ gives a path $\psi(\gamma)$
 between $1$ and $0$, which is followed in the homotopy.
 This is not a rational equivalence as only a subset of the points in a fiber $(\psi\circ\phi)^{-1}$ are followed along
 $\psi(\gamma)$ from 1 to 0 (these are the points above $\gamma\subset T$).
 \hfill$\diamond$
\end{remark}

\section{General witness sets}\label{S:Gwitness}

Let $X$ be an algebraic variety and fix \defcolor{$C_*$} to be an intersection theory as in
Proposition~\ref{P:intersectionTheory} such that $C_*X$ is a finitely generated free abelian group with 
nondegenerate intersection pairing.
A basis for $C_*X$ gives a normal form~\eqref{Eq:FundCyc} for a fundamental cycle $[V]$, leading to general witness 
sets.  
We discuss when general witness sets may be moved and may be used for sampling and membership.

\subsection{General witness sets}\label{SS:Gwitness}
The $k$th Betti number \defcolor{$b_k$} of $X$ is the rank of the free $\ZZ$-module $C_k X$.
While it has a $\ZZ$-basis of cycles $\alpha_1,\dotsc,\alpha_{b_k}\in Z_k X$, these need not be effective.
There are however, independent effective cycles $[L^{(k)}_1],\dotsc,[L^{(k)}_{b_k}]\in Z_k X$, with 
each $L^{(k)}_i$ an irreducible subvariety of dimension $k$.
These form a basis for the $\QQ$-vector space $C_k X\otimes_{\ZZ}\QQ$, called an \demph{effective $\QQ$-basis}. 
We work with a fixed choice of cycles $\{ L^{(a)}_b\}$ that form an effective $\QQ$-basis for $C_*X$. 

For a subvariety $V$ of $X$ of dimension $k$, there are rational numbers \defcolor{$c_j(V)$}
for $j=1,\dotsc,b_k$ defined by the expansion of the fundamental cycle of $V$ in this basis, 
 \begin{equation}\label{Eq:FundCyc}
   [V]\ =\ \sum_{j=1}^{b_k} c_j(V) [L^{(k)}_j]\,.
 \end{equation}
The intersection pairing on $C_{n-k}X \times C_{k}X$ is encoded by the $b_{n-k}\times b_{k}$ integer matrix $M^{(k)}$
whose entries are
 \begin{equation}\label{Eq:Intmatrix}
     \defcolor{M^{(k)}_{i,j}}\ :=\ \deg \bigl( [L^{(n-k)}_i] \cdot [L^{(k)}_j]\bigr)\,,
 \end{equation}
where $i=1,\dotsc,b_{n-k}$ and $j=1,\dotsc,b_k$.
As the intersection pairing is nondegenerate, $b_{n-k}=b_k$ and  $M^{(k)}$ is invertible.

Consequently, if $\defcolor{c(V)}:=(c_j(V)\mid j=1,\dotsc,b_k)^T$ is the vector of coefficients in the
representation~\eqref{Eq:FundCyc} of $V$, then the vector of intersection multiplicities,
\[
    \defcolor{d(V)}\ :=\ \bigl(\deg([V]\cdot[L^{(n-k)}_1]),\dotsc,\deg([V]\cdot[L^{(n-k)}_{b_{n-k}}])\bigr)^T\,,
\]
satisfies $d(V)= M^{(k)}c(V)$, and so we may recover the class~\eqref{Eq:FundCyc} of $V$ from these intersection
multiplicities by inverting this relation, $c(V)= (M^{(k)})^{-1}d(V)$.

As $\dim V+ \dim L^{(n-k)}_i=\dim X$, the product $[V]\cdot[L^{(n-k)}_i]$ is the image in $C_0 X$ of the
localized product $[V]\bullet[L^{(n-k)}_i]$ in $C_0(V\cap L^{(n-k)}_i)$.
This in turn is the image of the localized intersection product~\eqref{Eq:LIP} in $A_0(V\cap L^{(n-k)}_i)$ under the map
$A_*\to C_*$ of Proposition~\ref{P:intersectionTheory}.
When the intersection is proper (has dimension 0), $[V]\bullet[L^{(n-k)}_i]$ is the fundamental cycle
$[V\cap L^{(n-k)}_i]$ of the intersection, which is
\[
  \sum_{p\in V\cap L^{(n-k)}_i} m_p\, p\,,
\]
where $m_p$ is the intersection multiplicity of $V\cap L^{(n-k)}_i$ at $p$.

\begin{definition}
 Let $V\subset X$ be a subvariety of dimension $k$.
 A \demph{general witness set} for $V$ is a triple 
 \defcolor{$(V,\Lambda_\bullet,W_\bullet)$}, where 
 $\defcolor{\Lambda_\bullet}=(\Lambda_1,\dotsc,\Lambda_{b_{n-k}})$ are subvarieties of $X$ such that
 $[\Lambda_i]=[L^{(n-k)}_i]$ and $\defcolor{W_\bullet}=(W_1,\dotsc,W_{b_{n-k}})$ are cycles such that
 $\defcolor{W_i}\in Z_0(V\cap\Lambda_i)$ represents the localized product $[V]\bullet[\Lambda_i]$.
 We call each component $W_i$ a general witness set. \hfill$\diamond$
\end{definition}

By the preceeding discussion, general witness sets encode fundamental cycles.

\begin{theorem}
 Suppose that $(V,\Lambda_\bullet,W_\bullet)$ is a general witness set for $V$.
 The vector $c(V)$ of coefficients of\/ $[V]$ in the basis $[L_j^{(k)}]$  of $C_k X$ is obtained from the vector
 $\defcolor{\deg(W_\bullet)}:=(\deg(W_1),\dotsc,\deg(W_{b_{n-k}}))^T$ of the degrees of the $W_i$ 
 by the formula $c(V)=(M^{(k)})^{-1}\deg(W_\bullet)$.
\end{theorem}

\begin{example}
  The cycles $W_i$ are not necessarily effective.
  If $X:={\textit Bl}_p\PP^2$, the blow up of $\PP^2$ in a point $p$, then $C_1X=[\ell]\ZZ+[E]\ZZ$ (this holds in any
  intersection theory), where $\ell$ is the
  proper transform of a line in $\PP^2$ and $E$ is the exceptional divisor.
  In this case, $M^{(1)}=(\begin{smallmatrix}1&0\\0&-1\end{smallmatrix})$ as $[\ell]^2=1$, $[\ell]\cdot[E]=0$, and
  $[E]^2=-1$.
  A general witness set for $E$ is $W_\bullet=(0,-[q])$, where $q\in E$.\hfill$\diamond$
\end{example}

\begin{example}
 Projective space $\PP^n$ has free abelian Chow groups $A_*\PP^n$.
 Here, $b_k=1$ for $0\leq k\leq n$ and $L^{(k)}$ is any $k$-dimensional linear subspace ($k$-plane).
 By Bertini's Theorem, a general $(n{-}k)$-plane $\Lambda$ meets a $k$-dimensional subvariety $V$ of $\PP^n$
 transversally in $\deg(V)$ points $W=V\cap\Lambda$. 
 Thus classical witness sets are general witness sets.\hfill$\diamond$
\end{example}

Chow groups are not typically finitely generated free a\-be\-li\-an groups with a
nondegenerate intersection pairing---e.g.\ if $\calE$ is an elliptic curve, then $A_0\calE=\calE\times\ZZ$ and
$A_1\calE=\ZZ$. 
(This is remedied for $\calE$ by algebraic equivalence as $B_0\calE=B_1\calE=\ZZ$.)
Nevertheless, for many common varieties $X$, rational equivalence and numerical equivalence coincide.
A sufficient condition is that $X$ admits an action of a solvable linear algebraic group with finitely many
orbits~\cite{FMSS}. 
This class of varieties includes projective space, toric varieties, Grassmannians, flag manifolds, and 
spherical varieties.
If $X$ is such a space and $Y$ any variety, then there is a K\"unneth isomorphism
$A_*X\otimes A_*Y\xrightarrow{\,\sim\,}A_*(X\times Y)$, so products with these spaces preserve these
properties.

\begin{example}
 Hauenstein and Rodriguez~\cite{HR15} developed multiprojective witness sets for subvarieties of products of projective
 spaces, which are general witness sets for these varieties.
 Let $m,n\geq 1$.
 The Chow group of $\PP^m\times\PP^n$ is a free abelian group that is isomorphic to its cohomology.
 To describe a basis, for each $0\leq a\leq m$ and $0\leq b\leq n$, let $K_a\subset\PP^m$ and $L_b\subset\PP^n$ be
 linear subspaces of dimensions $a$ and $b$, respectively.
 The classes $[K_a]\otimes[L_b]=[K_a\times L_b]$ with $a+b=k$ form a basis for  $A_k(\PP^m\times\PP^n)$.

 A subvariety $V\subset\PP^m\times\PP^n$ of dimension $k$ has \demph{bidegrees} $\defcolor{d_{a,b}}=d_{a,b}(V)$ for
 $a+b=k$ defined by 
\[
    [V]\ =\ \sum_{a+b=k} d_{a,b}\, [K_a\times L_b]\,,
\]
 where $0\leq a\leq m$ and $0\leq b\leq n$.
 A \demph{multihomogeneous witness set} for $V$ is a triple $(V,\Lambda_\bullet,W_\bullet)$ where 
\begin{enumerate}
  \item[(i)]
    For each $(a,b)$ with $a+b=k$, $\Lambda_{a,b}=M^a\times N^b$ where $M^a\subset\PP^m$ and $N^b\subset\PP^n$ are
    linear subspaces of codimension $a$ and $b$, respectively, such that 

  \item[(ii)] $W_{a,b}:=V\cap\Lambda_{a,b}$ is transverse and therefore consists of $d_{a,b}$ points, and

  \item[(iii)] $\Lambda_\bullet=\{\Lambda_{a,b}\mid a+b=k\}$ and $W_\bullet=\{W_{a,b}\mid a+b=k\}$.
\end{enumerate}

 Hauenstein and Rodriguez enrich this structure by representing $V$ as a component of the
 solution set of a system of bihomogeneous polynomials and the linear subspaces $M^a$ and $N^b$ by general linear
 forms on their ambient projective spaces.
 They give algorithms based on multihomogeneous witness sets for moving a witness set, membership, sampling, regeneration, 
 and numerical irreducible decomposition using a trace test~\cite{SVW}.
 An alternative trace test for multihomogeneous witness sets is developed in~\cite{LRS}, and extensions to
 more than two factors are given in~\cite{toolkit}. \hfill$\diamond$
\end{example}

\subsection{Moving, sampling and membership}\label{S:SamplingMembership}

While general witness sets provide a representation of a cycle class $[V]$, without further assumption, their utility is
limited. 
We first describe how rational or algebraic equivalence allows a general witness set to be moved, and then discuss
conditions on subvarieties $\Lambda_i$ in an effective $\QQ$-basis that allow sampling and a membership test.
The moving lemma is essential for actual computations.

If $(V,\Lambda_\bullet,W_\bullet)$ is a general witness set for a subvariety $V\subset X$ of dimension $k$ and
$V\cap\Lambda_i$ is transverse, then we may move the general witness set $W_i$ using any elementary rational or algebraic
equivalence involving $\Lambda_i$.

\begin{theorem} \label{Th:MoveWitnessSet}
 Suppose that $\Lambda_i$ is an effective cycle with $[\Lambda_i]=[L_i^{(n-k)}]$ in $C_*X$ that meets a subvariety $V$
 transversally in a general witness set $W_i=V\cap\Lambda_i$.
 For any elementary rational equivalence $[\Lambda_i]-[\Lambda'_i]\in\Rat_{n-k}X$,
 homotopy continuation of $W_i$ along this rational equivalence as in Theorem~\ref{Th:RatEquivHomotopy}
 computes a general witness set $W'_i\in C_0(V\cap\Lambda'_i)$.
\end{theorem}

\begin{proof}
 Suppose that $Y\subset X\times\PP^1$ as in~\eqref{Eq:TwoProjections} gives an
 elementary rational equivalence $[\Lambda_i]-[\Lambda'_i]\in\Rat_{n-k}X$ (so that $\Lambda_i=\iota f^{-1}(1)$ and
 $\Lambda'_i=\iota f^{-1}(0)$).
 Since $V$ meets $\Lambda_i$ transversally in $W_i$, by Theorem~\ref{Th:RatEquivHomotopy}, there is a homotopy
 connecting $W_i$ with points on $V\cap \Lambda'_i$.
 If $V\cap \Lambda'_i$ is transverse, then numerical homotopy continuation may be used
 to compute the points $W'_i=V\cap \Lambda'_i$.
 If it is not transverse, so that homotopy paths become singular at $t=0$, then endgames~\cite{BHS11,HV98}
 may be used to compute the endpoints of the paths and the corresponding multiplicities.
 These points and multiplicities give a cycle $W_i'\in Z_0(V\cap \Lambda'_i)$ representing
 $[V]\bullet[\Lambda'_i]$.
\end{proof}

\begin{remark}  
 By Remark~\ref{R:AlgEquiv}, an elementary algebraic equivalence also gives a homotopy.\hfill$\diamond$
\end{remark}

The exceptional divisor $E$ of $X={\textit Bl}_p\PP^2$ does not move.
Thus it may not be possible to move a generator $L_i$ of
$C_{n-k}X$ in a rational or algebraic family, and thus move a general witness set $W_i$ for a subvariety $V$ of $X$.
Even when a generator $L_i$ moves, it may not move with sufficient freedom.

While $E\subset{\textit Bl}_p\PP^2$ does not move, the proper transform $\ell$ of a line moves fairly freely.
For any curve $C\subset{\textit Bl}_p\PP^2$ and any smooth point $x$ of $C$ with $x\not\in E$, there is a proper transform
$\ell'$ of a line in $\PP^2$ (so $[\ell]-[\ell']$ is an elementary rational equivalence) that contains $x$ and meets
$C\smallsetminus E$ transversally.
This suggests the following definition.

\begin{definition}
  A subvariety $\Lambda$ of $X$ satisfies the \demph{moving lem\-ma} with respect to a dense open subset $U$ of $X$ if
  for any subvariety $V$ of $X$ and smooth point $x\in V\cap U$, there is a subvariety $\Lambda'$ of $X$ containing
  $x$ with $V\cap\Lambda'$ transverse in $U$ and $[\Lambda]-[\Lambda']$ is an elementary rational
  equivalence.\hfill$\diamond$ 
\end{definition}

\begin{remark}
 Suppose that a member $L_i^{(n-k)}$ of an effective $\QQ$-basis for $C_{n-k}X$ satisfies the moving lemma with respect
 to $U$.
 Given a general witness set $W_i=V\cap L_i^{(n-k)}$, the algorithm of Theorem~\ref{Th:MoveWitnessSet} for
 moving $W_i$ may be used to sample points of $V\cap U$ and test membership in $V$ for points $x\in U$.\hfill$\diamond$
\end{remark}

It is always possible to choose an effective $\QQ$-basis for $C_{n-k}X$ with one member satisfying a generic moving lemma.

\begin{proposition}
  Let $X$ be a smooth variety and $V\subset X$ any affine open subset.
  Then there is a dense open subset $U\subset V$ such that 
  for every $k$ with $1\leq k\leq n=\dim X$, there exists a subvariety $\Lambda$ of $X$ of dimension $k$ that
  satisfies the moving lemma with respect to $U$.
\end{proposition}

\begin{proof}
  Let $\pi\colon V\to{\mathbb A}^n$ be a finite map given by Noether normalization
  and $U\subset V$ be the set of points $x$ where $d_x\pi$ is unrammified.
  Then the inverse images $\pi^{-1}(L)$ of affine $k$-planes $L$ in ${\mathbb A}^n$ form a family of rationally
  equivalent subvarieties which satisfy the moving lemma with respect to $U$. 
\end{proof}


\section{Schubert witness sets}\label{S:Schubert}


While an elementary rational equivalence gives rise to homotopies (Theorem~\ref{Th:RatEquivHomotopy}), the Chow
ring $A_*X$ of cycles on $X$ modulo rational quivalence does not typically satisfy hypotheses which allow general
witness sets as in Section~\ref{S:Gwitness}.
Even when $A_*X$ satisfies these hypotheses, a general witness set $W_i$ might not be an effective cycle or it might not
be possible to use $W_i$ for sampling or testing membership, even generically on an open subset $U\subset X$.

Nevertheless, for the important class of flag varieties, the theory of witness sets for subvarieties of
projective spaces extends optimally.
Flag varieties include projective spaces, Grassmannians, and products thereof.
The Chow ring of a flag variety has an integer basis of effective Schubert cycles, which are the fundamental
classes of Schubert varieties (defined below), and the intersection pairing is nondegenerate.
Consequently, subvarieties of $X$ have general witness sets.
Also, the intersection matrix $M^{k}$~\eqref{Eq:Intmatrix} is a permutation matrix, 
implying that the coefficients~\eqref{Eq:FundCyc} are positive integers.
Finally, each Schubert variety satisfies the moving lemma on the whole flag variety.
We expain all this below.

There is a well-known classification of flag varieties~\cite{Brion}.
Let \defcolor{$G$} be a semisimple reductive algebraic group, such as $SL_m\CC$, a symplectic or complex orthogonal group,
or a product of such groups.
A \defcolor{flag variety} for $G$ is a compact homogeneous space for $G$.
It has the form $G/P$ for $P$ a subgroup of $G$ containing a maximal solvable (Borel) subgroup $B$ of $G$.
The orbit of $B$ (or of any conjugate of $B$) on $G/P$ gives an algebraic cell decomposition of $G/P$.
Closures of these cells are Schubert varieties whose fundamental cycles give a $\ZZ$-basis for the Chow ring
$A_*G/P$.
This has a detailed combinatorial structure, which may be found in~\cite{Brion} or in~\cite{Fu97}, the latter for
$G=SL_m\CC$.
We summarize its salient features, which imply that the natural
general witness sets for flag varieties---Schubert witness sets---have the optimal properties of classical witness sets.
We describe them for the Grassmannian of lines in $\PP^4$, and show how to determine a Schubert witness set for the 
set of lines on a quadric $\PP^4$.

A \demph{partially ordered set} (\demph{poset}) is a set  $S$ with a binary relation $\leq$ that is
reflexive, antisymmetric, and transitive.
If $S$ has a minimal and a maximal element, and all maximal chains in $S$ have the same length, then $S$ is \demph{ranked}.
The rank \defcolor{$\rk(\alpha)$} of an element $\alpha\in S$ is the number of elements below $\alpha$ in any maximal chain
containing $\alpha$ and the rank of $S$ is the rank of its maximal element.
Write \defcolor{$S_k$} for the set of elements of $S$ of rank $k$.

We summarize some of the structure of a flag variety.
Proofs are found in~\cite{Brion,Fu97}.

\begin{proposition}\label{S:GmodS}
  For a flag variety $G/P$ of dimension $n$, rational equivalence coincides with numerical equivalence,
  and we have the following. 
 \begin{enumerate}
  \item[(i)] $G/P$ has an algebraic cell decomposition,
   \[
    G/P\ =\ \coprod_{\alpha\in S} X^\circ_\alpha\,,
   \]
    where $S$ is a ranked poset of rank $n$.
    We have $X^\circ_\alpha\simeq\CC^{\rk(\alpha)}$, and if \defcolor{$X_\alpha$} is the Zariski closure of $X^\circ_\alpha$, then 
   \[
    X_\alpha \ =\ \coprod_{\beta\leq \alpha} X^\circ_\beta\,.
   \]

  \item[(ii)] 
      We have $A_*G/P=\bigoplus_{\beta} [X_\beta]\cdot\ZZ$ and 
      \[A_kG/P\ =\ \bigoplus_{\rk(\beta)=k} [X_\beta]\cdot\ZZ\,,\]
       so that 
       $\{ [ X_\beta]\mid \beta\in S_k\}$ is a $\ZZ$-basis for $A_kG/P$.

  \item[(iii)]
      For any subvarieties $Y,Z\subset G/P$, there is a dense open subset $\calO$ of $G$ such that $gY\cap Z$ is
      generically transverse for $g\in\calO$.

  \item[(iv)] 
     For every $\beta\in S_k$, there is a $\hat{\beta}\in S_{n-k}$ and a dense
     open subset $\calO$ of $G$ such that for any 
     $\beta\in S_{n-k}$ and $g\in\calO$, the intersection 
     $g X_\alpha\cap X_\beta$ is empty if $\alpha\neq\hat{\beta}$, and if  $\alpha=\hat{\beta}$, then the intersection
     is transverse and consists of a single point.
 \end{enumerate}
\end{proposition}

\begin{remark}
  Part (iii), the moving lemma for subvarieties of $G/P$,
  is Kleiman's Bertini Theorem~\cite{KL74}.\hfill$\diamond$
\end{remark}

\begin{remark}
  When $G/P=\PP^n$, $S=[0,n]$ is a chain of length $n$ and $X_a$ is a linear subspace of dimension $a$.

  When $G/P=\PP^m\times\PP^n$, $S=[0,m]\times[0,n]$ and for $(a,b)\in S$,
  $X_{(a,b)}=K_a\times L_b$ where $K_a\subset\PP^m$ and $L_b\subset\PP^n$ are
  linear subspaces of dimensions $a$ and $b$, respectively.

  We explain this for the Grassmanian of lines in $\PP^4$ in \S~\ref{S:G14}.\hfill$\diamond$
\end{remark}

A flag variety $G/P$ has general witness sets, as rational and numerical equivalence coincide.
Using classes of Schubert varieties for a basis of $A_*G/P$, we obtain \demph{Schubert witness sets}.
For a subvariety $V\subset G/P$ of dimension $k$, a Schubert witness set $(V,gX_\bullet,W_\bullet)$ has the form
$gX_\bullet = (gX_\alpha\mid \alpha\in S_{n-k})$ with $g\in G$ and 
$W_\bullet = (W_\alpha\mid \alpha\in S_{n-k})$ where
 \begin{equation}\label{Eq:SWS}
  W_\alpha\ =\ V\cap g X_\alpha \quad\mbox{for }\alpha\in S_{n-k}
 \end{equation}
is transverse for all $\alpha\in S_{n-k}$.
By (iii) for each $\alpha$, a general translate $gX_\alpha$ meets $V$ transversally, and we may use the same group
element $g$ for every $\alpha\in S_{n-k}$.

By (iv), the intersection matrix $M^{(k)}$ is
\[
  M^{(k)}_{\alpha,\beta}\ =\ \left\{
    \begin{array}{rcl} 1&\ &\mbox{ if }\alpha=\hat{\beta}\\
      0&\ &\mbox{ if }\alpha\neq\hat{\beta}\end{array}\right.\ ,
\]
and thus
 \begin{equation}\label{Eq:SchubertRepresentation}
  [V]\ =\ \sum_{\beta\in S_{n-k}} \deg(W_{\hat{\beta}})[X_{\beta}]\ .
 \end{equation}

 We summarize some properties of Schubert witness sets.
 
\begin{theorem}
  Let $V\subset G/P$ be a subvariety of dimension $k$.
  Each component $W_\alpha$ of a Schubert witness set $W_\bullet$ is a multiplicity-free cycle.
  Any component $W_\alpha$ of a Schubert witness set~$\eqref{Eq:SWS}$ may be moved to any other Schubert witness set
  $W'_\alpha=V\cap hX_\alpha$ along an elementary rational equivalence.
  A non-zero Schubert witness set $W_\alpha$ may be used to sample points of $V$ and to test membership in $V$
  for any $x\in G/P$.
\end{theorem}

\begin{proof}
 For $\alpha\in S_{n-k}$, $W_\alpha=V\cap gX_\alpha$~\eqref{Eq:SWS} is transverse, so $W_\alpha$ is a multiplicity-free
 cycle. 
 Suppose that $W'_\alpha=V\cap hX_\alpha$ is the $\alpha$th component of another Schubert witness set for $V$.
 Let $\varphi\colon \CC\to G$ be a smooth rational map with $\varphi(0)=g$ and $\varphi(1)=h$ (e.g.\
 $\varphi(t)=\psi(t)g$ where $\psi$ is a one-parameter subgroup with $\psi(0)=1$ and $\psi(1)=hg^{-1}$).
 Then
 \[
    Y\ =\ \overline{\{(x,t)\mid x\in\varphi(t)X_\alpha\}}\ \subset\ G/P\times\PP^1
 \]
 as in~\eqref{Eq:TwoProjections} with $\iota f^{-1}(0)=gX_\alpha$ and $\iota f^{-1}(1)=hX_\alpha$.
 By Theorem~\ref{Th:RatEquivHomotopy}, $(V\times\PP^1)\cap Y$ is a homotopy between $W_\alpha$ and $W'_\alpha$.

 Since translates of $X_\alpha$ cover $G/P$ (and thus $V$), these homotopies may be used to sample points of $V$,
 and to test membership in $V$ for any $x\in G/P$ as in \S\S\ref{SS:witness}. 
\end{proof}

\begin{remark}
 Property (iv) of Proposition~\ref{S:GmodS}, that the Schubert basis is self-dual under the intersection pairing,
 simplifies the use of general witness sets. 
 A variety with an intersection theory $C_*$ that is finitely generated and has the property that every subvariety $V$
 satisfies the moving lemma for $U=X$ is a \demph{duality space} if the basis is self-dual under the intersection pairing
 as in $(iv)$.\hfill$\diamond$

General witness sets simplify when $X$ is a duality space.
If $\{ L_i^{(k)}\mid i=1,\dotsc,b_k\}$ are subvarieties whose cycles form a basis for $C_kX$ and
$\{ L_i^{(n-k)}\mid i=1,\dotsc,b_{n-k}=b_k\}$ subvarieties representing the dual basis in that
\[
   \deg \bigl( [ L_i^{(k)}]\bullet[L_j^{(n-k)}]\bigr)\ =\ 
     \left\{\begin{array}{rcl} 1&\ & i=j\\ 0&&\mbox{otherwise}
      \end{array}\right.
\]
Then if $V\subset X$ has dimension $k$ with general witness sets $W_i=V\cap\Lambda_i$, where the intersection is
transverse and $[\Lambda_i]=[L_i^{(n-k)}]$, then 
 \begin{equation}\label{Eq:NiceCycle}
   [V]\ =\ \sum \deg(W_i) \cdot [L_i^{(k)}]\,.
 \end{equation}
as in~\eqref{Eq:SchubertRepresentation}.\hfill$\diamond$
\end{remark}

\subsection{Schubert witness sets for $\GG(1,\PP^4)$}\label{S:G14}
Let \defcolor{$\GG(1,\PP^4)$} be the Grassmannian of lines in $\PP^4$.
This is a homogeneous space of dimension 6 for $SL_5\CC$.
Its Schubert decomposition is in terms of a flag of linear spaces 
\[
  \defcolor{M_\bullet}\ \colon\ M_0\ \in\ M_1\ \subset\ M_2\ \subset\ M_3\ \subset\ \PP^4\,,
\]
where $\dim M_i=i$.
Schubert varieties are parametrized by pairs $i,j$ with $0\leq i<j\leq 4$.
Then $\defcolor{X_{ij}}=\defcolor{X_{ij}M_\bullet}$ is 
\[
  \defcolor{X_{ij}M_\bullet}\ :=\ \{\ell\in \GG(1,\PP^4) \mid \emptyset\neq\ell\cap M_i,
  \ell\subset M_j\}\,.
\]
The dimension of $X_{ij}$ is $i+j-1$ and $X_{ij}\subset X_{ab}$ if $i\leq a$ and $j\leq b$.
Duality is given by $\widehat{ij}=4{-}j,4{-}i$.
We display the partial order $S$ for $\GG(1,\PP^4)$ below.
\[
  \begin{picture}(118,82)
   \put(2,2){\includegraphics{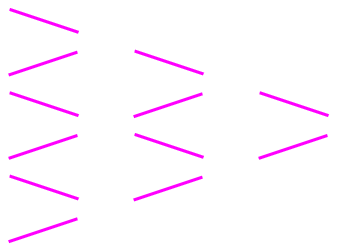}}
   \put(0,72){34}
         \put(36,60){24}
   \put(0,48){23}   \put(72,48){14}
         \put(36,36){13}   \put(108,36){04}
   \put(0,24){12}   \put(72,24){03}
         \put(36,12){02}
   \put(0,0){01}    
  \end{picture}
\]
Duality is obtained by reflecting in the horizontal line of symmetry with $\widehat{13}=13$ and $\widehat{04}=04$.

Let $\defcolor{Q}\subset\PP^4$ be a smooth quadric which is the zero set of a quadratic polynomial
$f$ and let $\defcolor{V_Q}\subset\GG(1,\PP^4)$ be the set of lines that lie on $Q$.
This has codimension 3 in $\GG(1,\PP^4)$.
Indeed, consider the parametrization $\ell(t)=tp+(1-t)q$ of the line through the points $p,q\in\PP^4$.
Then $f(\ell(t))$ is a quadratic polynomial in $t$ whose coefficients are polynomials in the coordinates of $p$ and $q$.
This line lies on $Q$ when the three coefficients of $f(\ell(t))$ vanish.

A Schubert witness set for $V_Q$ has the form
\[
  \bigl(V_Q, (W_{13},W_{04}), (gX_{13}, gX_{04}) \bigr)\,,
\]
where $W_\alpha=V_Q\cap gX_\alpha$ is transverse.
Let $M_\bullet$ be the flag in $\PP^4$ that defines $gX_\alpha$.
Since
\[
  X_{04}M_\bullet\ =\ \{\ell\mid M_0\in\ell\}
\]
is the set of lines that contain the point $M_0$ and $M_0\not\in Q$ (as $M_\bullet$ is general), we have
$W_{04}=V_Q\cap X_{04}M_\bullet=\emptyset$.
As
\[
  X_{13}M_\bullet\ =\ \{\ell \mid M_1\cap\ell\neq\emptyset\mbox{ and } \ell\subset M_3\}\,,
\]
we see that $V_Q\cap X_{13}M_\bullet$ is the set of lines $\ell$ on $Q\cap M_3$ that meet $M_1$.
Because $M_3$ is a general $\PP^3$, $Q\cap M_3$ is a quadratic hypersurface in $\PP^3$.
This contains two families of lines, and each point of $Q\cap M_3$ lies on one line from each family.
Since $M_1$ meets $Q$ in two points, $W_{13}=V_Q\cap X_{13}M_\bullet$ consists of four lines.
As
 \[
   [V_Q]\ =\ \deg(W_{13})[X_{13}] + \deg(W_{04})[X_{04}]\ =\  4[X_{13}]\,,
 \]
if $Q'$ is a second quadric, then
\[ 
  [V_Q\cap V_{Q'}]\ =\ [V_Q]^2\ =\ 16[X_{13}]^2\ =\ 16[X_{01}]\,,
\]
which recovers the well-known fact that 16 lines lie on a general quartic surface $Q\cap Q'$ in $\PP^4$. 

\providecommand{\bysame}{\leavevmode\hbox to3em{\hrulefill}\thinspace}
\providecommand{\MR}{\relax\ifhmode\unskip\space\fi MR }
\providecommand{\MRhref}[2]{%
  \href{http://www.ams.org/mathscinet-getitem?mr=#1}{#2}
}
\providecommand{\href}[2]{#2}

\end{document}